\documentclass{amsart}
\usepackage{amssymb}
\usepackage{stmaryrd} 
\usepackage{amsmath} 
\usepackage{amscd}
\usepackage{amsbsy}
\usepackage{commath}
\usepackage{comment, enumerate}
\usepackage{geometry}
\usepackage[matrix,arrow]{xy}
\usepackage{hyperref}
\usepackage{tabularx}

\usepackage{mathrsfs}
\usepackage{color}
\usepackage{mathtools,caption}
\usepackage{tikz-cd}
\usepackage{longtable}
\usepackage[utf8]{inputenc}
\usepackage[OT2,T1]{fontenc}
\usepackage[normalem]{ulem}
\usepackage{hyperref}
\hypersetup{
	colorlinks=true,
	linkcolor=blue,
	filecolor=violet,
	citecolor=orange,
	urlcolor=purple,
	pdftitle={The Iwasawa mu-invariant of certain elliptic curves of analytic rank zero},
}

\newtheorem{thm}{Theorem}[section]
\newtheorem*{thm*}{Theorem}

\newtheorem*{lem*}{Lemma}

\newtheorem{prop}[thm]{Proposition}
\newtheorem{cor}[thm]{Corollary}
\newtheorem{conj}[thm]{Conjecture}
\newtheorem{rmk}[thm]{Remark}

\theoremstyle{definition}
\newtheorem{example}{Example}

\def\Z{\mathbf{Z}}
\def\F{\mathbf{F}}
\def\R{\mathbf{R}}
\def\Q{\mathbf{Q}}
\def\C{\mathbf{C}}

\def\H{\mathbf{H}}

\def\E{\mathcal{E}}

\def\a{\alpha}
\def\b{\beta}

\def\w{\omega}

\def\z{\zeta}

\def\Qcyc{\Q_{\text{cyc}}}

\def\Zpt{\Z_p\llbracket T \rrbracket}
\def\Gal{\text{Gal}}

\DeclareFontFamily{U}{wncy}{}
\DeclareFontShape{U}{wncy}{m}{n}{<->wncyr10}{}
\DeclareSymbolFont{mcy}{U}{wncy}{m}{n}
\DeclareMathSymbol{\Sha}{\mathord}{mcy}{"58}

\def\sse{\subseteq}
\title{The Iwasawa $\mu$-invariant of certain elliptic curves of analytic rank zero}

\author[]{Adithya Chakravarthy}
\date{}
\begin{document}
	
	\maketitle	
	
	\begin{abstract}
     This paper is about the Iwasawa theory of elliptic curves over the cyclotomic $\Z_p$-extension $\Qcyc$ of $\Q$. We discuss a deep conjecture of Greenberg that if $E/\Q$ is an elliptic curve with good ordinary reduction at $p$, and $E[p]$ is irreducible as a Galois module, then the Selmer group of $E$ over $\Qcyc$ has $\mu$-invariant zero. We prove new cases of Greenberg's conjecture for some elliptic curves of analytic rank $0$. The proof involves studying the $p$-adic $L$-function of $E$. The crucial input is a new technique using the Rankin-Selberg method.   
	\end{abstract}

    \section{Introduction}

    We begin with a fundamental theorem of Iwasawa, which serves as the starting point of Iwasawa theory.  Let $K$ be a number field and let $K_{\infty}/K$ be a $\Z_p$-extension: a Galois extension with Galois group isomorphic to the additive group $\Z_p$ of $p$-adic integers. For each integer $n > 0$ there is a unique subfield $K_n$ of $K_{\infty}$ of degree $p^n$ over $K$. Iwasawa proved the following now famous theorem about the growth of class numbers in such towers.	 
	\begin{thm*}[Iwasawa]
		Let $K$ be a number field and let $K_{\infty}/K$ be a $\Z_p$ extension with layers $K_n$.  Suppose that $p^{e_n}$ is the exact power of $p$ dividing the class number of $K_n$. Then there exist integers $\mu, \lambda, \nu$  such that 
		\[ e_n = \mu p^n + \lambda n + \nu \]
		for all sufficiently large values of $n$.
	\end{thm*}	
	The most important example of a $\Z_p$-extension of $K$ is the so-called \textit{cyclotomic $\Z_p$-extension} of $K$, denoted $K^{\text{cyc}}/K$. It is defined by letting $K^{\text{cyc}}$ be the appropriate subfield of $\cup_{n \geq 1} K(\z_{p^n})$. Iwasawa has conjectured the following about the cyclotomic $\Z_p$-extension.
	\begin{conj}[Iwasawa's $\mu=0$ Conjecture] Let $K$ be a number field and let $K^{\text{cyc}}$ be the cyclotomic $\Z_p$ extension of $K$. Then $\mu=0$. 
	\end{conj}
	This conjecture is probably the deepest open problem from classical Iwasawa theory. The only progress towards this comes from a paper of Ferrero and Washington \cite{fw}, where they prove the conjecture in the case when $K/\Q$ is abelian. Other than this, the conjecture remains wide open and it seems intractable at present.

	This paper is about a generalization of Iwasawa's $\mu=0$ conjecture to the Iwasawa theory of \textit{elliptic curves}. We now briefly outline the main aspects of the Iwasawa theory of elliptic curves and then state the analogue of Iwasawa's $\mu=0$ conjecture in this context. In a landmark paper \cite{mazur}, Mazur observed that many features of classical Iwasawa theory could be used to study elliptic curves. This initiated the study of what we now call the Iwasawa theory of elliptic curves. If $E$ is an elliptic curve over a number field $K$, let $\Sha(E/K)[p^{\infty}]$ denote the $p$-primary part of the Tate-Shafarevich group of $E$ over $K$. We then have the following analogue \cite[Theorem 1.10]{greenberg} of Iwasawa's theorem.
	\begin{thm*}
		Let $\Q^{\text{cyc}} = \cup \Q_n$ denote the cyclotomic $\Z_p$ extension of $\Q$. Let $E/\Q$ be an elliptic curve, and let $p$ be a prime of good ordinary reduction. Assume that $\Sha(E/\Q_n)[p^{\infty}]$ is finite for all $n$ and let $p^{e_n}$ be the exact power of $p$ dividing $|\Sha(E/\Q_n)[p^{\infty}]|$. Then there exist integers $\mu, \lambda, \nu$  such that 
		\[ e_n = \mu p^n + \lambda n + \nu \]
		for all sufficiently large values of $n$.
	\end{thm*} 
	The quantity $\mu_p(E)$ in the above formula is called the \emph{$\mu$-invariant} of $E$ over $\Q^{\text{cyc}}$. (See Section 2 for a more abstract definition of $\mu_p(E)$.) In analogy with Iwasawa's $\mu=0$ conjecture for class groups, one might expect that $\mu_p(E)$ always vanishes. This turns out to be \textit{false}. Indeed, Mazur \cite{mazur} discovered that the elliptic curve $X_0(11)$ has a nonzero $\mu$-invariant at the prime $p=5$. However, in a beautiful paper \cite[Conjecture 1.11]{greenberg}, Greenberg conjectures that one can rescue Iwasawa's conjecture in the following sense: 
	
	\begin{conj} Let $E/\Q$ be an elliptic curve, let $p$ be an odd prime of good ordinary reduction. Then there is a $\Q$-isogenous curve $E'$ such that $\mu_p(E') = 0$.  
	\end{conj}

	That is, Greenberg predicts that $\mu$ can be made zero after perhaps shifting by an isogeny. In particular, the above conjecture implies:
	
	\begin{conj}[Greenberg's $\mu=0$ Conjecture] 
		\label{greenberg}
		Let $E/\Q$ be an elliptic curve, and let $p$ be an odd prime of good ordinary reduction. If $E[p]$ is \textit{irreducible} as a Galois module, then $\mu_p(E)=0$.  
	\end{conj}
	
	\begin{rmk}
		It is known by the work of Drinen \cite{drinen} that this conjecture is false over general number fields. So it applies particularly for curves $E$ over $\Q$. 
	\end{rmk}
	
	\begin{rmk}
		Schneider has given a simple formula for the effect of an isogeny on the $\mu$-invariant of $E$ for odd $p$ (see \cite[Isogeny Formula, Second Form]{schneider}). Thus, Greenberg's conjecture effectively predicts the $\mu$-invariant of any elliptic curve over $\Q$.
	\end{rmk}
	
	Greenberg's conjecture bifurcates into two cases, one of which seems to be much harder than the other. First, there is the case where $E[p]$ is reducible as a Galois module (the "reducible case") and second, there is the case where $E[p]$ is irreducible as a Galois module (the "irreducible case"). Most of the progress to date has been in the \textit{reducible} case, where we have the following theorem \cite[Theorem 1.3]{greenberg-vatsal} of Greenberg and Vatsal: if $E[p]$ has a $\Gal(\overline{\Q}/\Q)$-submodule which is either ramified at $p$ and even, or unramified at $p$ and odd, then $\mu_p(E)=0$. This theorem follows from a bootstrapping of the Ferrero-Washington Theorem. This result was furthered by Trifkovic \cite{mak}, who showed that Greenberg's conjecture was true for infinitely many curves $E$ with $E[p]$ reducible, for $p = 3$ or $5$. The core strategy in Trifkovic's paper was an explicit evaluation of the global duality pairing for finite flat group schemes over rings of integers.

    There has also been progress in proving that "$\mu=0$" can be propagated in families of modular forms. In a beautiful paper, Emerton-Pollack-Weston \cite{emerton-pollack-weston}, showed that "$\mu=0$" is invariant in Hida families. That is, if $f$ is a modular form with $\mu=0$, then that implies that every form in the Hida family passing through $f$ also has $\mu=0$. On another front, Ray recently \cite{ray2} showed the following result for $p=5$: there is a positive proportion of elliptic curves $E/\Q$ for which $\mu_5(E) = \lambda_5(E)=0$. Ray uses results of Bhargava-Shankar \cite{bhargava-shankar} on average ranks of Selmer groups. In another paper, Ray \cite{ray1} also showed that if $E[p]$ is irreducible, if the classical Iwasawa $\mu$-invariant vanishes for the splitting field $\Q(E[p])$, and a purely Galois theoretic condition on $E[p]$ holds, Greenberg's conjecture should hold.

    The irreducible case of Greenberg's conjecture, however, remains unsolved in general. In this paper, we prove new results about the irreducible case of Greenberg's Conjecture, specifically for some elliptic curves with analytic rank $0$: 
	
	\begin{thm}
        \label{thm: main-thm-elliptic-curves}
		Let $E/\Q$ be an elliptic curve of conductor $N$. Let $p$ be a prime of good ordinary reduction. Assume:
		\begin{enumerate}
			\item $L(E/K, 1) \neq 0$, where $K = \Q(\sqrt{-N})$ and $L(E/K,s)$ denotes the Hasse-Weil $L$-series of $E$ over $K$, 
			\item $p > 100 \sqrt{N} (\log N + 5) m_E$, where $m_E$ denotes the modular degree of $E$, 
                \item $p$ does not divide the the Manin constant of $E$, 
                \item $p$ is not anomalous for $E$, i.e: $p$ does not divide $\#E(\F_p)$,
			\item $N$ is a prime congruent to $3$ modulo $4$. 
		\end{enumerate} 
		Then $\mu_p(E) = 0$ and $\lambda_p(E)=0$.  
	\end{thm}

    The following are well-known consequences: 

    \begin{cor}
    \label{cor: sha-vanishes}
        Keep the notations and assumptions as above. Let $\Qcyc$ denote the cyclotomic $\Z_p$-extension of $\Q$ and let $\Q_n$ denote the $n$-th layer of $\Qcyc$. Then: 
        \begin{enumerate}
            \item the Mordell-Weil rank of $E(\Q_n)$ is zero for all $n \geq 1$, and 

            \item the $p$-primary part of $\Sha(E/\Q)$ vanishes.
        \end{enumerate}
    \end{cor}

    \subsection{Examples} Let $E_1 = \texttt{11a1}$ be the elliptic curve of conductor $11$. We show in the final section of the paper that $L(E_1/\Q(\sqrt{-11}),1) \neq 0$. Theorem~\ref{thm: main-thm-elliptic-curves} says that if $p > 100 \sqrt{N} (\log N + 5) \, m_E \approx 2453.604 \dots$ is a good ordinary prime, then $\mu_p(E_1) = \lambda(E_1) = 0$. Furthermore, we have $\Sha(E/\Q)[p^{\infty}] = 0$ and the rank of $E$ over $\Q_n$ is zero for all $n \geq 1$.

    As another example, let $E_2 = \texttt{67a1}$ be the elliptic curve of conductor $67$. We show in the final section of the paper that $L(E_2/\Q(\sqrt{-67}),1) \neq 0$. If $p > 100 \sqrt{N} (\log N + 5) m_E \approx 37671.828 \dots$ is a good ordinary prime which is not anomalous for $E_2$, then $\mu_p(E_2) = \lambda_p(E_2)=0$. Furthermore, we have $\Sha(E_2/\Q)[p^{\infty}] = 0$ and the rank of $E_2$ over $\Q_n$ is zero for all $n \geq 1$.

    However, some lower bound on $p$ is necessary. Indeed, consider the elliptic curve $E_3 = \texttt{24691a1}$ and $p = 5$. From the LMFDB \cite{lmfdb}, the analytic order of $\Sha(E_3/\Q)$ is $25$. So the conclusion of Corollary~\ref{cor: sha-vanishes} is false in this case. That being said, it should definitely be possible to improve the bound $p \geq 100 \sqrt{N} (\log N + 5) \, m_E$ in Theorem~\ref{thm: main-thm-elliptic-curves}, but we have not attempted to do so here.
    
    \subsection{Strategy} Using results of Kato on the Iwasawa main conjecture, it is known that in order to prove that $\mu_p(E) = \lambda_p(E)= 0$, it suffices to show that the $p$-adic $L$-function of $E$ is a unit. One knows that if $p$ is not anomalous and that $p$ does not divide $L(E,1)/\Omega_E$, then the $p$-adic $L$-function is a unit. So as long as one avoids the finitely many prime divisors of $L(E,1)/\Omega_E$ (and the anomalous primes), then we have $\mu_p(E)=0$ and $\lambda_p(E)=0$. But a priori, one has no idea what primes divide $L(E,1)/\Omega_E$. The main point of this paper is to give an \textit{explicit upper bound} for which primes divide $L(E,1)/\Omega_E$. This is roughly the statement of Theorem~\ref{thm: main-thm-modular-forms}. We prove Theorem~\ref{thm: main-thm-modular-forms} using the Rankin-Selberg method.  

     \begin{rmk}
        In proving Theorem~\ref{thm: main-thm-elliptic-curves}, we actually show Greenberg's Conjecture both for $E$ and the quadratic twist $E^{(-N)}$ by $-N$.  
    \end{rmk}

    \subsection{Outline}

    In the first section, we explain the Iwasawa main conjecture. We then reduce Theorem~\ref{thm: main-thm-elliptic-curves} to Theorem~\ref{thm: main-thm-modular-forms}, which is purely a statement about $L$-values of modular forms. In the next section, we explain the Rankin-Selberg method. In the section after that, we prove Theorem~\ref{thm: main-thm-modular-forms}. Finally in the last section, we give numerical examples that use the Rankin-Selberg method to calculate Iwasawa invariants.

    \subsection{Acknowledgements} I would like to thank Kumar Murty and members of the GANITA lab for their constant support. I would also like to thank Debanjana Kundu for teaching me Iwasawa theory, for suggesting this problem to me, and for her mentorship. I would also like to thank Antonio Lei and Robert Pollack for providing helpful comments on a preliminary draft of this paper and for suggesting further applications. 
    
    \section{Preliminaries}
    
    In this section, $p$ will denote an odd prime and $E/\Q$ will be an elliptic curve with good ordinary reduction at $p$. Let $\Q^{\text{cyc}}$ denote the cyclotomic $\Z_p$ extension of $\Q$. 
	
	To fix notation, we briefly summarize some facts about the Iwasawa main conjecture for elliptic curves. We will first summarize the objects on the algebraic side of the main conjecture, then summarize the objects on the analytic side, and then state the Iwasawa main conjecture which bridges these two worlds.
	
	\subsection*{The algebraic side}

	For any algebraic extension $K/\Q$, the Selmer group of $E$ over $K$ is a certain subgroup of $H^1(G_K, E(\overline{\Q})_{\text{tors}})$, where $G_K = \Gal(\overline{K}/K)$.
	The Selmer group fits into the  fundamental exact sequence
	\[ 0 \to E(K) \otimes \Q/\Z \to \text{Sel}(E/K) \to \Sha(E/K) \to 0 \] 
	where $\Sha(E/K)$ denotes the Tate-Shafarevich group of $E$ over $K$. Let $K = \Q^{\text{cyc}}$. Then we can consider the Selmer group $\text{Sel}(E/\Q^{\text{cyc}})$, and this has an action of $\Gamma=\Gal(\Q^{\text{cyc}}/\Q)$. Its $p$-primary subgroup  $\text{Sel}(E/\Q^{\text{cyc}})_p$ can be regarded as a $\Lambda$-module, where $\Lambda = \Zpt$. This ring $\Lambda$ is called the \textit{Iwasawa algebra}. It is now known from the deep work of Kato \cite{kato} that the Pontryagin dual $X(E/\Q^{\text{cyc}}) = \text{Sel}(E/\Q^{\text{cyc}})_p^{\vee}$ is a finitely generated \textit{torsion} $\Lambda$-module. Therefore the structure theorem of finitely generated $\Lambda$-modules (see for example \cite[Theorem 13.12]{washington}) says that one has a pseudo-isomorphism
	\[ X(E/\Q^{\text{cyc}}) \sim \left( \oplus_{i=1}^n \Lambda/(f_i(T)^{a_i}) \right) \oplus \left( \oplus_{j=1}^m \Lambda/(p^{\mu_j}) \right),  \] 
	where the $f_i(T)$'s are irreducible distinguished polynomials\footnote{A polynomial $f(T)$ is \textit{distinguished} if when you reduce $f$ modulo $p$, only the highest degree term remains.} in $\Lambda$.  One can then define the algebraic Iwasawa invariants by 
	\[ \lambda_{E}^{\text{alg}} = \sum_{i=1}^n a_i \text{deg}(f_i(T)), \hspace{20pt}\text{and}\hspace{20pt} \mu^{\mathrm{alg}}_p(E) = \sum_{j=0}^m \mu_j. \]
	The \textit{characteristic ideal} of $X(E/\Q^{\text{cyc}})$ is the ideal of $\Lambda$ generated by $p^{\mu}f_1(T)^{a_1}\dots f_n(T)^{a_n}$. 
	
	\subsection*{The analytic side}
	For an elliptic curve $E/\Q$ with good ordinary reduction at a prime $p$ and $\chi$ an even Dirichlet character, denote by $L(E, \chi, s)$ the Hasse Weil $L$-function of $E$ twisted by $\chi$. Let $H_1(E(\C), \Z)^{\pm}$ be the eigenspaces under complex conjugation, and let $\Lambda_E^{\pm}$ be the respective generators of these spaces over $\Z$. The periods $\Omega_E^+$ and $\Omega_E^-$ are defined by 
    \[ \Omega_E^{\pm} = \int_{\Lambda_E^{\pm}} \w_E, \]
    where $\w_E$ is the invariant differential of $E$. We may assume that $\Omega_E^+$ and $i\Omega_E^-$ are real positive. It is known by the work of Shimura that $L(E, \chi, 1)/\Omega_E^{\chi(-1)}$, \textit{a priori} a transcendental number, is in fact an \textit{algebraic} number. 
	Mazur and Swinnterton-Dyer have attached to $E$ a $p$-adic $L$-function $\mathcal{L}_p(E/\Q, T) \in \Lambda \otimes \Q_p$ satisfying the following interpolation properties. If we write $a_p = (p+1) - \#E(\F_p)$, consider the Hecke polynomial $X^2 - a_pX + p$. Let $\a \in \Z_p^{\times}$ denote this unique $p$-adic unit root of the Hecke polynomial. Then,
    \begin{equation}
        \label{eqn: interpolation-property}
        \mathcal{L}_p(E/\Q, 0) = \left( 1- \dfrac{1}{\a} \right)^2 \cdot \dfrac{L(E,1)}{\Omega_E^+}.
    \end{equation}
	
	Let $\chi$ be an even Dirichlet character of conductor $p^n$ and $p$-power order. Then,
	\[ \mathcal{L}_p(E/\Q, \chi(1+p)-1) = \dfrac{1}{\a^{n+1}} \cdot \dfrac{p^{n+1}}{g(\chi^{-1})} \dfrac{L(E,\chi^{-1}, 1)}{\Omega_E^+}. \]
	Using the Weierstrauss preparation theorem \cite[Theorem 7.3]{washington}, we can define the analytic invariants $\mu^{\mathrm{an}}_p(E)$ and $\lambda_E^{an}$ by writing:
	\[  \mathcal{L}_p(E/\Q, T) = p^{\mu^{\mathrm{an}}_p(E)} \cdot u(T) \cdot f(T) \]
	where $f(T)$ is a distinguished polynomial of degree $\lambda_E^{an}$ and $u(T)$ is a unit in $\Lambda$.
	It is known (see \cite[Proposition 3.7]{stein-wuthrich}) that if $E/\Q$ has good ordinary reduction at an \textit{odd }prime $p$, and that $E[p]$ is irreducible as a Galois module, then $\mu^{\mathrm{an}}_p(E) \geq 0$.
	In other words, $\mathcal{L}_p(E/\Q, T) \in \Lambda$. 
	
	\subsection*{Iwasawa Main Conjecture}
	The Iwasawa \textit{main conjecture} relates the Selmer group on the algebraic side to the $p$-adic $L$-function on the analytic side. Precisely, suppose that $E/\Q$ is an elliptic curve with good ordinary reduction at an \textit{odd} prime $p$, and that $E[p]$ is irreducible as a Galois module. Then on the algebraic side, we can look at the  characteristic ideal of the $p$-primary Selmer group $X(E/\Q^{\text{cyc}})$; this is an ideal in $\Lambda$. On the analytic side, we can attach to $E$ a $p$-adic $L$-function  $\mathcal{L}_p(E/\Q, T) \in \Lambda$.  The \textit{main conjecture} asserts that the characteristic ideal of the $p$-primary Selmer group $X(E/\Q^{\text{cyc}})$ is generated by $\mathcal{L}_p(E/\Q, T)$ in $\Lambda$. In particular, it implies that $\mu_{E}^{\text{alg}}=\mu_{E}^{\text{an}}$ and $\lambda_{E}^{\text{alg}}=\lambda_{E}^{\text{an}}$. Kato \cite{kato} has proven a deep result on one divisibility of the main conjecture, which is the starting point of our proof.
	
	\begin{thm}[Kato]
		\label{kato}
		Let $E/\Q$ be an elliptic curve, and let $p \geq 5$ be a prime of good ordinary reduction. Suppose that the mod $p$ Galois representation $\overline{\rho_{E,p}}: G_{\Q} \to \text{GL}_2(\F_p)$ is surjective. Then the characteristic ideal of $X(E/\Q^{\text{cyc}})$ divides the ideal generated by $\mathcal{L}_p(E/\Q, T)$ in $\Lambda$. In particular, $\mu_{E}^{\text{alg}}\leq\mu_{E}^{\text{an}}$ and $\lambda_{E}^{\text{alg}}\leq \lambda_{E}^{\text{an}}$.
	\end{thm} 
	
    \subsection{Modular Forms}
    For the rest of the paper, fix an embedding $\overline{\Q} \hookrightarrow \overline{\Q_p}$. Kato's Theorem shows that to prove $\mu^{\text{alg}}_p(E) = \lambda^{\text{alg}}_p(E) = 0$, it suffices to show that $\mu^{\text{an}}_p(E) = \lambda^{\text{an}}_p(E) = 0$. This is equivalent to showing that the constant term of $\mathcal{L}_p(E/\Q, T)$ is nonzero mod $p$; by (\ref{eqn: interpolation-property}), this means
    \[ \left( 1- \dfrac{1}{\a} \right)^2 \cdot \dfrac{L(E,1)}{\Omega_E^+} \not\equiv 0 \mod p. \]
    (Note that we are viewing $L(E,1)/\Omega_E^+$ as an element of $\overline{\Q_p}$ via our fixed emebedding $\overline{\Q} \hookrightarrow \overline{\Q_p}$.) This is really a problem about modular forms, not elliptic curves. To see this, let $f \in S_2(\Gamma_0(N))$ be the newform associated to $E$ by modularity. Then we have an equality of $L$-functions: $L(E,s) = L(f,s)$. Define the Petersson norm of $f$ as follows:
    \[ \langle f ,f \rangle = \int_{ \Gamma_0(N) \setminus \H} |f|^2 dx \, dy.  \]
    An elliptic curve $E/\Q$ is \textit{optimal} if the map $H_1(X_0(N)(\C), \Z) \to H_1(E(\C), \Z)$ is surjective. (See \cite{stevens-inventiones}). We will derive Theorem~\ref{thm: main-thm-elliptic-curves} from: 
    
	\begin{thm}
        \label{thm: main-thm-modular-forms}
		Let $N$ be a prime and let $f \in S_2(\Gamma_0(N))$ be a weight $2$ newform with rational Fourier coefficients. Let $p$ be an odd prime. Assume that 
        \begin{enumerate}
            \item $N$ is congruent to $3$ modulo $4$, 
			\item $L(f, 1)L(f,\chi,1) \neq 0$, where $\chi$ is the unique quadratic character mod $N$, and 
			\item $p > 100 \sqrt{N} (\log N + 5) \deg \pi_E$, where $E$ is an optimal elliptic curve in the isogeny class corresponding to $f$. 
			 
		\end{enumerate} 
		Then $p$ does not divide $\dfrac{L(f,1) \, L(f, \chi, 1)}{\pi^2 \langle f,f \rangle} \, \dfrac{g(\chi)}{i} $.
	\end{thm}

    \begin{proof}[Proof of Theorem~\ref{thm: main-thm-modular-forms} $\implies$ Theorem~\ref{thm: main-thm-elliptic-curves}]
        Let $E/\Q$ be an elliptic curve of conductor $N$ as in Theorem~\ref{thm: main-thm-elliptic-curves}. There exists a modular parametrization $\pi_E: X_0(N) \to E$ that is a surjective morphism over $\Q$. Let $c_E$ be the Manin constant of $E$, that is, the number such that 
	\[ \pi_E^*(\w_E) = c_E \cdot 2\pi i f(z) dz,\]
	where $\w_E$ is the invariant differential on $E$ and $f \in S_2(\Gamma_0(N))$ is the modular form attached to $E$. Let $m_E = \deg \pi_E$ denote the modular degree of $E$. We have 
	\begin{equation}
 \label{eqn: petersson-period}
		\Omega_E^+ \Omega_E^- 
		= \int_{E(\C)} \w_E \wedge  \overline{\w_E} \\
		= \dfrac{4\pi^2 c_E^2}{m_E} \langle f,f \rangle,
	\end{equation}
    where $\Omega_E^{\pm}$ are the real and imaginary Neron periods of $E$. Recall that we have assumed in Theorem~\ref{thm: main-thm-elliptic-curves} that $p$ does not divide $m_E$ or $c_E$, and that $p > 100 \sqrt{N} (\log N + 5) m_E$. Therefore, the assumptions of Theorem~\ref{thm: main-thm-modular-forms} are satisfied. We conclude that $p$ does not divide 
    \[ \dfrac{L(E,1)}{\Omega_E^+}  \cdot \dfrac{L(E, \chi, 1)}{\Omega_E^-}\,  \, \dfrac{g(\chi)}{i}. \]
    It follows that $p$ does not divide $L(E,1)/\Omega^+_E$. By the interpolation property (\ref{eqn: interpolation-property}) of the $p$-adic $L$-function of $E$, if $p$ is not anomalous for $E$, we conclude that the constant term of $\mathcal{L}_p(E,T)$ is a $p$-adic unit and hence the analytic $\mu$ and $\lambda$-invariants of $E$ vanish. Kato's Theorem implies that the algebraic $\mu$ and $\lambda$-invariants of $E$ vanish as well. This proves Theorem~\ref{thm: main-thm-elliptic-curves}.
    \end{proof}    

    \begin{rmk}
        Theorem~\ref{thm: main-thm-modular-forms} works equally well in the supersingular case and the ordinary case. In the supersingular case, Kurihara \cite{kurihara} has proven an explicit growth formula for the Tate Shafarevich group in the cyclotomic tower, assuming that $p$ does not divide $L(E,1)/\Omega_E$. Theorem~\ref{thm: main-thm-modular-forms} gives an explicit upper bound for which primes divide $L(E,1)/\Omega_E$, which in turn gives an upper bound for primes $p$ for which where Kurihara's theorem applies. (We thank Robert Pollack for pointing out this application to us.)
    \end{rmk}
	
	\section{Rankin-Selberg Method}
	
    The rest of the paper is devoted to proving Theorem~\ref{thm: main-thm-modular-forms}. The proof of Theorem~\ref{thm: main-thm-modular-forms} was catalyzed by reading Shimura's \cite{shimura} paper on the Rankin-Selberg method, which we now explain.
    
    Let $N \geq 1$ be an integer and let $\w$ be an odd Dirichlet character modulo $N$. Put
	\[ G_{\w}(\tau,s) = \dfrac{N}{-4\pi i g(\w)}\sum_{(m,n) \in \Z^2 \setminus (0,0)} \dfrac{\w(n)}{(mN\tau + n) |mN\tau + n|^{2s}}, \]
    where $g(\w)$ is the Gauss sum. This function has a meromorphic continuation to all $s \in \C$ (see \cite[Page 788]{shimura} and the references therein) and it is holomorphic at $s=0$. Thus we can set 
	\[ G_{\w}(\tau) \coloneqq G_{\w}(\tau,0). \] 
	By \cite[Page 788]{shimura}, the function $G_{\w}(\tau)$ belongs to $M_1(\Gamma_0(N), \w)$. If $\w$ is \textit{primitive}, by \cite[Equation (3.4)]{shimura}, we have the Fourier expansion\footnote{Shimura expresses the constant term of $G_{\w}(\tau)$ in terms of the value $L(1,\w)$, whereas we have applied the functional equation for $L(s,\w)$ to express the constant term in terms of the value $L(0,\w)$.}
	\begin{equation}
		\label{eqn: G_chipsi-q-expansion}
		G_{\w}(\tau) = \dfrac{L(0, \omega)}{2} + \sum_{n=1}^{\infty} \left( \sum_{d \vert n} \w(d) \right) q^n. 
	\end{equation}
	
	Let us now fix an element $f \in S_2(\Gamma_0(N))$ and an element $g \in M_1(\Gamma_0(N), \w)$, where $\w$ is a Dirichlet character modulo $N$. Suppose that they have Fourier expansions
	\[ f(\tau) = \sum_{n=1}^{\infty} a_n q^n, \hspace{20pt} g(\tau) = \sum_{n=0}^{\infty} b_n q^n.  \]
	Then we put 
	\[ L(s,f \times g) \coloneqq L(2s-1, \omega) \sum_{n=1}^{\infty} \dfrac{a_n b_n}{n^s}, \]
	where $L(s,\w)$ is the usual Dirichlet $L$-function attached to $\w$. The $L$-function $L(s,f \times g)$ is called the \textit{Rankin-Selberg convolution $L$-function}\footnote{Shimura uses the notation $D(s,f,g)$ instead of the now more standard $L(s, f\times g)$.} of $f$ and $g$. For an integer $M \geq 1$ and modular forms $g \in S_k(\Gamma_0(M))$ and $h \in M_k(\Gamma_0(M))$, define the Petersson inner product as 
	\[ \langle g,h \rangle \coloneqq \int_{\Gamma_0(M) \setminus \H} \overline{g(\tau)} h(\tau) y^{k-2} \, dx \, dy. \]
	
	Then by \cite[Equation (2.4)]{shimura}, we have 
	\begin{equation}
		\label{eqn1}
		\langle f, g \cdot G_{\overline{\w}} \rangle = L(1,f \times g) \dfrac{g(\overline{\w})}{8\pi^2 i},
	\end{equation}
	where $g(\overline{\w})$ is the Gauss sum.  To determine the form $g$, we need the following
	
	\begin{prop}
		\label{prop: G_chi_psi-q-expansion}
		Let $\chi, \psi$ be primitive Dirichlet characters modulo $M_1$ and $M_2$, respectively, such that $\chi\psi(-1)=-1$. There exists an element $G_{\chi, \psi} = \sum_{n=0}^{\infty} a_n q^n \in \E_1(\Gamma_0(M_1M_2), \chi\psi)$(= space of weight $1$ Eisenstein series for $\Gamma_0(M_1M_2)$ and nebentypus $\chi\psi$) such that 
		\[ a_n = \sum_{d \vert n} \chi(d) \psi(n/d), \hspace{20pt}n \geq 1  \]
		and 
		\[ a_0 = 
		\begin{cases}
			0 &\text{ if both } \chi, \psi \text{ are non-trivial,}\\
			\dfrac{L(0, \chi\psi)}{2} &\text{ otherwise}.
		\end{cases}
		\]
	\end{prop}
	
	\begin{proof}
		See for example \cite[Theorem 4.7.1]{miyake}.
	\end{proof}
	
	Now suppose that $f \in S_2(\Gamma_0(N))$ is a newform. Let $\chi, \psi$ be Dirichlet characters such that $\chi\psi(-1)=-1$. By \cite[Equation (4.3)]{shimura}, we have 
	\begin{equation}
		\label{eqn2}
		L(1, f \times G_{\chi, \psi}) = L(f, \chi, 1) L(f, \psi, 1),
	\end{equation}  
	where $G_{\chi, \psi}$ is the function from Proposition~\ref{prop: G_chi_psi-q-expansion} above. Combining (\ref{eqn1}) and (\ref{eqn2}) and dividing through by $\langle f,f \rangle$, we obtain the main result of this section:
	
	\begin{prop}
		\label{prop: shimura-identity}
		Let $N \geq 1$ be an integer and let $f \in S_2(\Gamma_0(N))$ be a newform. Let $\chi, \psi$ be primitive Dirichlet characters such that $\chi\psi(-1)=-1$. Then:
		\[\dfrac{\langle f, G_{\chi, \psi} \cdot G_{\overline{\chi\psi}} \rangle}{\langle f,f \rangle} = \dfrac{L(f, \chi, 1) L(f, \psi, 1)}{\langle f,f \rangle} \cdot \dfrac{g(\overline{\chi\psi})}{8\pi^2i},    \]
    where $g(\overline{\chi\psi})$ denotes the Gauss sum.
	\end{prop}

	\section{Proof of Theorem~\ref{thm: main-thm-modular-forms}}
	
	We first need two preliminary results.
		\begin{prop}
		\label{prop: L(0,chi)-inequality}
		Let $N$ be an odd integer. Let $\chi$ be an odd primitive quadratic character of conductor $N$. Then 
		\[ L(0, \chi) \leq \frac{\sqrt{N}}{2 \pi} ( \log N + 5). \]
	\end{prop}
	
	\begin{proof}
		In \cite[Corollary 2]{ramare}, Ramar{\'e} proves the following bound for odd primitive characters $\chi$ modulo odd $N$:
        \begin{equation}
            \label{eqn: ramare}
            |(1 - \chi(2)/2) L(1,\chi)| \leq \frac{1}{4} ( \log N + 5). 
        \end{equation}
        (Note that Ramar{\'e} does not require $\chi$ to be quadratic.) The functional equation for $L(s, \chi)$ gives
        \[ L(0, \chi) = \dfrac{g(\chi)}{i \pi} L(1, \chi). \]
        Since $|g(\chi)| = \sqrt{N}$ for $\chi$ quadratic, (\ref{eqn: ramare}) gives us 
        \[ L(0, \chi) \leq \frac{\sqrt{N}}{2 \pi} ( \log N + 5 - 2 \log (3/2)) \leq \frac{\sqrt{N}}{2 \pi} ( \log N + 5),   \]
        which completes the proof.        
	\end{proof}

    Next we show that algebraic $L$-values have bounded denominators. An elliptic curve $E/\Q$ is \textit{optimal} if the map $H_1(X_0(N)(\C), \Z) \to H_1(E(\C), \Z)$ is surjective.  
    
    \begin{prop}
        \label{prop: L-value-denominators}
        Let $N$ be a prime which is $3$ mod $4$. Let $f \in S_2(\Gamma_0(N))$ be a newform and let $\chi$ be an odd quadratic character mod $N$. Assume that $L(f,1) \, L(f, \chi, 1) \neq 0$. Then the rational number  
        \begin{equation}
            \label{eqn: L-value-integrality}
            \dfrac{L(f,1) \, L(f, \chi, 1)}{\pi^2 \langle f_j,f_j \rangle} \, \dfrac{g(\chi)}{i}
        \end{equation}
        has denominator $\leq 576 m_E$, where $E$ is an optimal curve in the isogeny class of $f$ and $m_E$ denotes the modular degree of $E$. 
    \end{prop}

    The proof of Proposition~\ref{prop: L-value-denominators} has a distinct thread from the rest of the proof so we prove it in the next section.

	We are now ready to prove Theorem~\ref{thm: main-thm-modular-forms}. Let $N$ be a prime which is $3$ mod $4$. Let $\chi$ be the unique quadratic character mod $N$. Since $N \equiv 3$ mod $4$, the character $\chi$ is odd. From Proposition~\ref{prop: shimura-identity}, we have the identity:
	\begin{equation}
		\label{eqn: spectral-expansion-quadratic}
		G_{1, \chi} \cdot G_{\overline{\chi}} = cE + \sum_{j} \left( \dfrac{L(f_j,1) \, L(f_j, \chi, 1)}{\pi^2 \langle f_j,f_j \rangle} \, \dfrac{g(\chi)}{i} \right) f_j,
	\end{equation}
	where the sum runs over all newforms $f_j \in S_2(\Gamma_0(N))$ and $E = \dfrac{N-1}{12} + q + \dots \in \E_2(\Gamma_0(N))$ is the unique normalized Eisenstein series and $c \in \C$ is a constant. Looking at the first Fourier coefficient (i.e: the coefficient of $q$) on each side, we obtain:
	\[ L(0, \chi) = c + \sum_{j} \dfrac{L(f_j,1) \, L(f_j, \chi, 1)}{\pi^2 \langle f_j,f_j \rangle} \, \dfrac{g(\chi)}{i}.  \]
	It is classical that $g(\chi) = i\sqrt{N}$. Dividing both sides by $g(\chi)/i = \sqrt{N}$ and rearranging, we obtain:
	\begin{equation}
		\label{eqn: key-equality}
		\dfrac{L(0,\chi) - c}{\sqrt{N}} = \sum_{j} \dfrac{L(f_j,1) \, L(f_j, \chi, 1)}{\pi^2 \langle f_j,f_j \rangle}.
	\end{equation}
	From (\ref{eqn: spectral-expansion-quadratic}) we deduce that 
	\[ c = \dfrac{\text{constant term of }G_{1, \chi}^2}{\text{constant term of }E} = \dfrac{L(0, \chi)^2 / 4}{(N-1)/12} > 0. \]
	Combining this with the estimate $L(0, \chi) < \frac{\sqrt{N}}{2 \pi} ( \log N + 5)$ from Proposition~\ref{prop: L(0,chi)-inequality}, we find that 
	\begin{equation}
		\label{eqn: L(0,chi)-estimate}
		\dfrac{L(0,\chi) - c}{\sqrt{N}} < \dfrac{L(0,\chi)}{\sqrt{N}} < \frac{1}{2 \pi} (\log N + 5). 
	\end{equation}
	
	Now comes the key step: we appeal to a deep theorem of Guo \cite{guo} about non-negativity of $L$-values.
	\begin{thm}[Guo]
		\label{thm: guo}
		Let $N \geq 1$ be an integer and let $f \in S_2(\Gamma_0(N))$ be an eigenform. Let $\chi$ be a quadratic character. Then $L(f,1) \geq 0$ and $L(f, \chi, 1) \geq 0$. 
	\end{thm}
	Theorem~\ref{thm: guo} implies that for each newform $f_j \in S_2(\Gamma_0(N))$, we have 
	\[ \dfrac{L(f_j,1) \, L(f_j, \chi, 1)}{\pi^2 \langle f_j,f_j \rangle} \leq \sum_{j} \dfrac{L(f_j,1) \, L(f_j, \chi, 1)}{\pi^2 \langle f_j,f_j \rangle}. \]
	Combining this with (\ref{eqn: L(0,chi)-estimate}) and (\ref{eqn: key-equality}), we conclude that 
	\[  \dfrac{L(f_j,1) \, L(f_j, \chi, 1)}{\pi^2 \langle f_j,f_j \rangle} < \frac{1}{2 \pi} (\log N + 5) \]
	for each newform $f_j \in S_2(\Gamma_0(N))$. Multiplying through by $g(\chi)/i = \sqrt{N}$, we obtain:
	\begin{equation}  
		\label{eqn: eqn1}
		\dfrac{L(f_j,1) \, L(f_j, \chi, 1)}{\pi^2 \langle f_j,f_j \rangle} \, \dfrac{g(\chi)}{i}  <  \frac{\sqrt{N}}{2 \pi} (\log N + 5).
	\end{equation}
	Now suppose that $f = f_j$ has rational Fourier coefficients. By Proposition~\ref{prop: L-value-denominators}, the LHS of (\ref{eqn: eqn1}) is a rational number with denominator $\leq 576 m_E$, where $m_E$ is the modular degree of $E$. By assumption, we have $L(f,1) \, L(f, \chi, 1) \neq 0$. So we can write the LHS of (\ref{eqn: eqn1}) in the form $a/n$ for some $a > 0$, where $n \leq 576 m_E$. Since $a/n < \frac{\sqrt{N}}{2 \pi} (\log N + 5)$ by (\ref{eqn: eqn1}), we conclude that $a< 576 \frac{\sqrt{N}}{2 \pi} (\log N + 5) m_E$. So if $p$ is a prime such that 
    \[ p \geq 100 \sqrt{N} (\log N + 5) m_E > 576 \cdot \frac{\sqrt{N}}{2 \pi} (\log N + 5) m_E, \]
    then $p$ cannot divide $a$. This implies that 
	\[ p \not \vert \left( \dfrac{L(f,1) \, L(f, \chi, 1)}{\pi^2 \langle f,f \rangle} \, \dfrac{g(\chi)}{i} \right), \] 
    completing the proof of Theorem~\ref{thm: main-thm-modular-forms}, and thus Theorem~\ref{thm: main-thm-elliptic-curves}.

    \section{Proof of Proposition~\ref{prop: L-value-denominators}}

    In this section, we follow closely the template of a paper by Wiersema and Wuthrich \cite{wuthrich-wiersema}. Let $E/\Q$ be an elliptic curve of conductor $N$. Let $f$ be the newform associated to $E$. For $r \in \Q$, define
    \begin{equation}
        \label{eqn: lambda-r}
        \lambda(r) = 2\pi i \int_{i\infty}^r f(z) \, dz = \dfrac{1}{c_E} \int_{\gamma(r)} \w_E,
    \end{equation}
    where $\gamma(r)$ is the image in $E(\C)$ of the vertical line in the upper half plane from $i \infty$ to $r$, and $c_E$ is the Manin constant of $E$. Let $\Lambda_E$ be the Neron lattice of $E$, i.e. the set of all values $\int_{\gamma} \w_E$ as $\gamma$ runs through closed loops in $E(\C)$. Let $c_{\infty}$ be the number of connected components of $E(\R)$. Then we have (see \cite[Section 2]{wuthrich-wiersema}):
    \begin{equation}
        \label{eqn: neron-lattice-1}
        \{ \mathrm{Re}(z): z \in \Lambda_E \} = \dfrac{1}{2} \Omega_E^+ \Z 
    \end{equation} 
    and 
    \begin{equation}
        \label{eqn: neron-lattice-2}
        \{ \mathrm{Im}(z)i : z \in \Lambda_E \} = \dfrac{1}{2}c_{\infty} \Omega_E^- \Z \sse \dfrac{1}{2} \Omega_E^- \Z.
    \end{equation} 
    
    We now define the standard modular symbols $[r]^{\pm}$ by
    \[ [r]^+ = \dfrac{\mathrm{Re}(\lambda(r))}{\Omega_E^+} \hspace{20pt} \mathrm{and} \hspace{20pt} [r]^- = \dfrac{\mathrm{Im}(\lambda(r))}{\Omega_E^-} \]
    for any $r \in \Q$. For any Dirichlet character $\chi$ mod $m$, we have
    \begin{equation}
        \label{eqn: birch-lemma}
        \dfrac{L(E, \chi, 1) g(\overline{\chi})}{\Omega_E^{\chi(-1)}} = \sum_{a \,\mathrm{mod}\, m} \chi(a) \left[ \frac{a}{m} \right]^{\chi(-1)}.
    \end{equation}

    \begin{proof}[Proof of Proposition~\ref{prop: L-value-denominators}]
        Let $\pi_E: X_0(N) \to E$ denote the modular parametrization of $E$. Since $N$ is prime, the elliptic curve $E$ is semistable. Therefore, one can show that for all $r \in \Q$, the point $P_r = \pi_E(r)$ is a torsion point defined over $\Q$ (see \cite[Theorem 1.3.1]{stevens}). By Mazur's theorem on torsion \cite{mazur-eisenstein-ideal}, the point $P_r$ must have order dividing $1, \dots, 10$, or $12$.
        Hence the path $\gamma$ in the definition of $\lambda(r)$ ends at a point in $E[m]$ for some $m = 1, \dots, 10, 12$. So $\lambda(r) \in \frac{1}{m} \Lambda$. Now by (\ref{eqn: neron-lattice-1}) and (\ref{eqn: neron-lattice-2}) we obtain $[r]^{\pm} \in \frac{1}{2m} \Z$. By (\ref{eqn: birch-lemma}), it follows that for all Dirichlet characters $\chi$, the value $2m \cdot \frac{L(E, \chi, 1) g(\overline{\chi})}{\Omega_E^{\chi(-1)}}$ is an algebraic integer for some $m = 1,\dots, 10, 12$. (The $m$ does not depend on $\chi$.) 

        By (\ref{eqn: petersson-period}), we have 
        \begin{equation}
            \label{eqn: petersson-comparison}
            \dfrac{L(f, 1) L(f, \chi, 1)}{\pi^2 \langle f,f \rangle} \dfrac{g(\chi)}{i} = \dfrac{4c_E^2}{\deg \pi_E} \dfrac{L(E, 1)}{\Omega_E^+} \dfrac{L(E, \chi, 1)}{\Omega_E^-} \dfrac{g(\chi)}{i}.
        \end{equation}   
        Since $E$ is a optimal curve, there is a conjecture that the Manin constant $c_E$ is equal to $1$. And this conjecture is known to be true if $E$ is semistable by \cite{manin-constant}, so $c_E = 1$. It follows that the denominator of (\ref{eqn: petersson-comparison}) divides $4m^2 \deg \pi_E$ for some $m = 1,\dots, 10,12$. Since $m \leq 12$, the denominator of (\ref{eqn: petersson-comparison}) is $\leq 4 \cdot 12^2 \deg \pi_E = 576 \deg \pi_E$.     
    \end{proof}
	
\section{Examples}
	
	The Rankin-Selberg method allows us to compute the algebraic parts of $L$-values, which in turn allows us to calculate Iwasawa invariants. We demonstrate this in two examples.
	
	\begin{example}
		Let $N = 11$ and let $\chi$ be the unique (odd) quadratic character of conductor $11$. The Eisenstein series $G_{\mathbf{1}, \chi}$ is the unique eigenform (normalized so that $a_1 = 1$) in the space $\E_1(\Gamma_0(11), \chi)$:
		\[ G_{\mathbf{1}, \chi} = \frac{1}{2} +  q + 2 q^{3} +  q^{4} + 2 q^{5} + \dots. \]
	The space $M_2(\Gamma_0(11))$ is spanned by an Eisenstein series and a cuspidal newform whose $q$-expansions are:
		\begin{align*}
			E &= \frac{5}{12} +  q + 3 q^{2} + 4 q^{3} + 7 q^{4} + 6 q^{5} + \dots, \\
			f &= q - 2 q^{2} -  q^{3} + 2 q^{4} +  q^{5} + \dots.
		\end{align*}
		By examining these $q$-expansions, we find that 
        \begin{equation*}
            G_{\mathbf{1}, \chi}^2  = \dfrac{3}{5} E + \dfrac{2}{5}f.
        \end{equation*}	Proposition~\ref{prop: shimura-identity} then says that 
		\begin{equation*}
    	\dfrac{L(f, 1) L(f, \chi, 1)}{\langle f,f \rangle} \cdot \dfrac{g(\chi)}{8\pi^2i} = \dfrac{\langle f, G_{\mathbf{1}, \chi}^2 \rangle}{\langle f,f \rangle} = \dfrac{2}{5}.
		\end{equation*}
        Let $E = X_0(11)$ be the elliptic curve corresponding to $f$. Let $p \neq 2,5$ be a prime of good ordinary reduction which is not anomalous for $E$. Then $2/5$ is a $p$-adic unit so $\mu_p(E) = \lambda_p(E) = 0$ for all such primes $p$. 
	\end{example}

        \begin{example}
		Let $N = 67$ and let $\chi$ be the unique (odd) quadratic character of conductor $67$. The Eisenstein series $G_{\mathbf{1}, \chi}$ has $q$-expansion  
		\[ G_{\mathbf{1}, \chi} = \dfrac{1}{2} + q + q^2 + q^4 + 2q^5 + \dots \]
	The space $M_2(\Gamma_0(67))$ is spanned by an Eisenstein series and five newforms $f_1, f_2, f_2^c, f_3, f_3^c$, where $f_i^c$ denotes the Galois conjugate of $f_i$: 
		\begin{align*}
			E &= \frac{11}{4} +  q + 3 q^{2} + 4 q^{3} + 7 q^{4} + 6 q^{5} + \dots, \\
			f_1 &= q + 2 q^{2} -  2q^{3} + 2 q^{4} +  2q^{5} + \dots. \\
                f_2 &= q + (-\b - 1)  q^2 + (\b - 2)  q^3 + 3\b  q^4 - 3  q^5 +  \dots. \\
                f_3 &= q - \b  q^2 + (-\b + 1)  q^3 + (\b - 1)  q^4 + (2\b + 1)  q^5 + \dots, \\
		\end{align*}
        where $\b = \frac{1}{2}(1+\sqrt{5})$. We have  
        \begin{equation*}
            G_{\mathbf{1}, \chi}^2  = \dfrac{1}{11} E + \dfrac{2}{5}f_1 + 0(f_2 + f_2^c) + \dfrac{2}{55}(7-\sqrt{5}) f_3 + \dfrac{2}{55}(7+\sqrt{5}) f_3^c.
        \end{equation*}
        Proposition~\ref{prop: shimura-identity} then says that 
		\begin{equation*}
    	\dfrac{L(f_1, 1) L(f_1, \chi, 1)}{\langle f_1,f_1 \rangle} \cdot \dfrac{g(\chi)}{8\pi^2i} = \dfrac{\langle f_1, G_{\mathbf{1}, \chi}^2 \rangle}{\langle f_1,f_1 \rangle} = \dfrac{2}{5}.
		\end{equation*}
        Let $E$ be the elliptic curve of conductor $67$ corresponding to $f_1$. Let $p \neq 2,5$ be a prime which is not anomalous for $E$. Then $2/5$ is a $p$-adic unit so $\mu_p(E) = \lambda_p(E) = 0$ for all such primes $p$. 
	\end{example}

	\bibliographystyle{amsalpha}
	\bibliography{bibliography}
	
\end{document}